\documentclass[smallcondensed]{svjour3}

\journalname{Journal of Scientific Computing}

\usepackage{xcolor}
\usepackage{amsmath,amssymb,amsthm}
\usepackage{enumitem}
\usepackage{stackrel}
\usepackage{lineno}
\usepackage{ulem}
\usepackage[T1]{fontenc}
\usepackage{booktabs}
\usepackage{hyperref}
\usepackage{diagbox}
\usepackage{booktabs}
\usepackage{longtable}
\usepackage{mathtools}
\usepackage{adjustbox}
\usepackage[section]{placeins}
\usepackage{tabu}

\newcommand{\pderivative}[2]{\frac{\partial #1}{\partial #2}}

\newcommand\iprod[1]{\left\langle #1\right\rangle} 			
\newcommand\inorm[1]{\left |\left| #1\right|\right|}			
\newcommand\iprodN[1]{\left\langle #1\right\rangle_{\!N}}	
\newcommand\inormN[1]{\left |\left| #1\right|\right|_{N}}		

\newcommand\Nvec[1]{\underline{#1}}
\newcommand\Dmat{\mathcal{D}}
\newcommand\Mmat{\mathcal{M}}
\newcommand\Vmat{\mathcal{V}}
\newcommand\Bmat{\mathcal{B}}
\newcommand\Vinvmat{\mathcal{V}^{-1}}
\newcommand\Cmat{\mathcal{C}}
\newcommand\quadmat{\mathcal{K}}
\newcommand\Filtermat{\mathcal{F}}
\newcommand\explicitFiltermat{\widetilde{\Filtermat}_{\text{aux}}}

\numberwithin{equation}{section}

\newtheoremstyle{dotless}{}{}{\itshape}{}{\bfseries}{}{ }{}
\theoremstyle{dotless}
\newtheorem{rem}{Remark}

\usepackage{xargs}    
\usepackage[colorinlistoftodos,prependcaption,textsize=tiny]{todonotes}
\newcommandx{\unsure}[2][1=]{\todo[linecolor=blue,backgroundcolor=blue!25,bordercolor=blue,#1]{#2}}
\newcommandx{\changeThis}[2][1=]{\todo[linecolor=red,backgroundcolor=red!25,bordercolor=red,#1]{#2}}
\newcommandx{\info}[2][1=]{\todo[linecolor=OliveGreen,backgroundcolor=OliveGreen!25,bordercolor=OliveGreen,#1]{#2}}
\newcommandx{\improvement}[2][1=]{\todo[linecolor=gray,backgroundcolor=gray!25,bordercolor=gray,#1]{#2}}
\newcommandx{\thiswillnotshow}[2][1=]{\todo[disable,#1]{#2}}

\begin{document}
%
\title{Stable filtering procedures for nodal discontinuous Galerkin methods}
\titlerunning{Explicit stable in time filtering for nodal DG methods}
\author{Jan Nordstr\"{o}m \and Andrew R. Winters}
\institute{J. Nordstr\"{o}m \at Department of Mathematics; Computational Mathematics, Link\"{o}ping University, SE-581 83 Link\"{o}ping, Sweden\\
Department of Mathematics and Applied Mathematics, University of Johannesburg, P.O. Box 524, Auckland Park 2006, Johannesburg, South Africa\\
\email{jan.nordstrom@liu.se} \and A. R. Winters \at Department of Mathematics; Computational Mathematics, Link\"{o}ping University, SE-581 83 Link\"{o}ping, Sweden\\ \email{andrew.ross.winters@liu.se}}
\date{Received: date / Accepted: date}

\maketitle

\begin{abstract}
We prove that the most common filtering procedure for nodal discontinuous Galerkin (DG) methods is stable. The proof exploits that the DG approximation is constructed from polynomial basis functions and that integrals are approximated with high-order accurate Legendre-Gauss-Lobatto quadrature. The theoretical discussion serves to re-contextualize stable filtering results for finite difference methods into the DG setting. It is shown that the stability of the filtering is equivalent to a particular contractivity condition borrowed from the analysis of so-called transmission problems. As such, the temporal stability proof relies on the fact that the underlying spatial discretization of the problem possesses a semi-discrete bound on the solution. Numerical tests are provided to verify and validate the underlying theoretical results.
\end{abstract}


\section{Introduction}

High frequency errors are always present in numerical simulations due to inaccuracies of the spatial derivatives at high-wave numbers. These errors can be exacerbated due to aliasing errors or steep gradients in the solution which can exist (or arise) for solutions of partial differential equations (PDEs) with predominantly hyperbolic character. To combat such errors, one technique is to use a filter operator that removes high wave number oscillations, e.g. \cite{chaudhur2017,Gassner:2013qf,Hestahven:1008th,meister2012}. The filtering procedure is separate from the scheme itself, and is often done in an ad hoc fashion (filtering is applied ``as little as possible, but as much as needed'').

There exist a multitude of numerical methods to approximate the solution of hyperbolic time-dependent PDEs. The family of discontinuous Galerkin (DG) spectral methods is well-suited for hyperbolic problems due to its high-order nature and ability to propagate waves accurately over long times \cite{Airnsworth2004}. In particular, nodal collocation DG methods are attractive because of their computational efficiency \cite{Kopriva:2009nx}. Additionally, if a nodal DG method is constructed on the Legendre-Gauss-Lobatto (LGL) nodes then the discrete differentiation and integration operators satisfy a summation-by-parts (SBP) property \cite{kreiss1974finite,kreiss1977,strand1994,svard2014} for any polynomial order \cite{gassner_skew_burgers}. This allows for semi-discrete stability estimates to be constructed for such high-order nodal DG methods, e.g \cite{carpenter_esdg,Gassner:2013ij}.

Recently, Lundquist and Nordstr\"{o}m \cite{lundquist2020stable} removed the ad hoc nature of filtering in the context of finite difference methods. Therein, they discuss a contractivity condition on the explicit filter matrix by re-framing the filtering procedure as a transmission problem \cite{nordstrom2018well}. Further, the work in \cite{lundquist2020stable} develops a necessary condition on the existence of an auxiliary filter matrix and how it must be related to the particular discrete integration (quadrature). Also, an implicit implementation of filtering was proved to be stable, i.e. sufficiency was obtained. The explicit implementation was more easily implemented, slightly more accurate and numerically shown to be stable but a proof was not obtained. 


The goal of the present work is to re-interpret the theoretical work from \cite{lundquist2020stable} into the nodal DG context. In doing so, we will remove the ad hoc nature of the nodal DG filtering procedure and prove that the commonly used explicit filter technique from the spectral community \cite{don1994numerical,hesthaven2008filtering,vandeven1991} is stable in time. Just as in the case of finite difference methods, the temporal stability of the filtering for nodal DG relies upon the existence of a high-accuracy auxiliary filter matrix as well as a semi-discrete bound on the solution.

The remainder of the paper is organized as follows: Section \ref{sec:DGOverview} provides an overview of the nodal DG method and the commonly used filtering procedure. Then, Section \ref{sec:stableFilterProof} generalizes the theoretical time stability results from \cite{lundquist2020stable} into the DG context. Numerical results that support and verify the theoretical findings are given in Section \ref{sec:numResults}. Our concluding remarks and outlook are given in the final section. 

\section{Overview of nodal DG approximations}\label{sec:DGOverview}

Discontinuous Galerkin methods are principally designed to approximate solutions of hyperbolic conservation laws \cite{Hestahven:1008th,Kopriva:2009nx}. Here, we consider such a conservation law in one spatial dimension
\begin{equation}\label{eq:linearAdvection}
\pderivative{u}{t} + \pderivative{f}{x} = 0,\qquad x \in [x_L,x_R],
\end{equation}
where $u\equiv u(x,t)$ is the solution and $f\equiv f(u)$ is the flux function. The conservation law is then equipped with an initial condition $u(x,0)\equiv u_{\text{ini}}(x)$ and suitable boundary condition(s). Two prototypical examples of conservation laws are the linear advection equation and Burgers' equation whose corresponding flux functions are
\begin{equation}
\text{linear advection: }\, f(u) = a(x,t) u,\qquad\text{Burgers': }\, f(u) = \frac{u^2}{2}.
\end{equation}

Next, we provide the broad strokes to arrive at the nodal DG approximation of the conservation law \eqref{eq:linearAdvection}. First, we transform the domain $[x_L,x_R]$ to the reference interval $[-1,1]$. To do so, we apply the mapping
\begin{equation}
\xi(x) = 2\left(\frac{x-x_L}{x_R-x_L}\right) - 1,\quad\textrm{such that}\quad \xi\in[-1,1],
\end{equation}
and rewrite the conservation law in the computational coordinate:
\begin{equation}\label{eq:mappedLinAdv}
\pderivative{u}{t} + \pderivative{f}{x} = 0 \qquad\Rightarrow\qquad  \frac{\Delta x}{2}\pderivative{u}{t} + \pderivative{f}{\xi} = 0.
\end{equation}

\subsection{The numerical scheme}

The DG approximation is built from the weak form of the mapped equation \eqref{eq:mappedLinAdv}. We list the nodal DG approximation steps below with full details given in \cite{Kopriva:2009nx}:  
\begin{enumerate}
\item Multiply by a test function $\varphi$ and integrate over the reference domain.
\item Integrate-by-parts once and resolve discontinuities at the physical boundaries with a numerical flux function $f^*(u^L,u^R)$.
\item Integrate-by-parts again to obtain \textit{strong form DG}. 
\item Approximate the solution and flux with nodal polynomials of degree $N$ written in the Lagrange basis, e.g.
\begin{equation}
u(x,t) \approx U(x,t) = \sum_{j=0}^NU_j(t)\ell_j(\xi)
\end{equation}
where the interpolation nodes are taken to be the $N+1$ Legendre-Gauss-Lobatto (LGL) nodes.
\item Select the test function to be the Lagrange polynomials $\varphi = \ell_i(\xi)$ with $i=0,\ldots,N$.
\item Approximate integrals with LGL quadrature such that the DG scheme is collocated. 
\item Arrive at the semi-discrete approximation that can be integrated in time.
\end{enumerate}

The resulting strong form, nodal DG approximation is then
\begin{equation}\label{eq:semiDG}
\frac{\Delta x}{2}\dot{\Nvec{U}} + \Dmat\,\Nvec{F} + \Mmat^{-1}\Bmat\left(\Nvec{F}^* - \Nvec{F}\right) = 0,
\end{equation}
where
\begin{equation}\label{eq:vecNotation}
\Nvec{U} = \left[U_0,U_1,\ldots,U_N\right]^T
\end{equation}
is the vector form for the degrees of freedom. The matrices in the semi-discrete form \eqref{eq:semiDG} are the discrete derivative matrix \cite{Kopriva:2009nx}
\begin{equation}\label{eq:discDer}
\Dmat_{ij} = \ell'_j(\xi_i),\quad i,j = 0,\ldots,N,
\end{equation}
the matrix $\Mmat$ containing the LGL quadrature weights and the boundary matrix $\Bmat$ given by
\begin{equation}\label{eq:massBndyMats}
\Mmat = \text{diag}(\omega_0,\ldots,\omega_N)\qquad\text{and}\qquad\Bmat = \text{diag}(-1,0,\ldots,0,1).
\end{equation}
There exist several numerical flux functions depending on the particular mathematical flux function $f(u)$, see Toro \cite{toro2009} for details. A general (and simple) numerical flux which we use in the present work is the local Lax-Friedrichs flux
\begin{equation}
F^*(U^L,U^R) = \frac{1}{2}\left(F^L + F^R\right) - \frac{\lambda_{\max}}{2}\left(U^R - U^L\right),
\end{equation}
where $\lambda_{\max}$ is the maximum wave speed of the flux Jacobian.

Two features to note for this flavour of nodal DG approximation are: The diagonal mass matrix $\Mmat$ denotes a quadrature rule that is exact for polynomials up to degree $2N-1$. So, there is equality between the continuous and discrete integral
\begin{equation}\label{eq:equalityDiscCont}
\iprodN{V,W} = \Nvec{V}^T\,\Mmat\,\Nvec{W} = \sum_{i=0}^N V_i\Mmat_{ii}W_i = \int\limits_{-1}^1 v\,w\,\mathrm{d}\xi = \iprod{v,w}
\end{equation}
provided the product of the functions $v$ and $w$ are polynomials of degree $\leq 2N-1$. 
From \eqref{eq:equalityDiscCont} the continuous and discrete $L_2$ norms are denoted
\begin{equation}\label{eq:discNorm}
\iprod{v,v} = \inorm{v}^2\quad\text{and}\quad\iprodN{V,V} = \inormN{V}^2.
\end{equation}
The mass and derivative matrices of the LGL collocation DG scheme \eqref{eq:semiDG} form a summation-by-parts (SBP) operator pair \cite{kreiss1974finite,kreiss1977,strand1994,svard2014} for any nodal polynomial order $N$ \cite{gassner_skew_burgers}, i.e.
\begin{equation}\label{eq:SBP}
\Mmat\,\Dmat + (\Mmat\,\Dmat)^T = \Bmat.
\end{equation}
From the SBP property \eqref{eq:SBP}, stable versions of the nodal DG method can be constructed, e.g. \cite{carpenter_esdg,chan2018,gassner_skew_burgers,Gassner:2016ye}.
\subsection{Construction of filtering for nodal DG}\label{sec:standardDG}

In the context of DG methods, the general idea of filtering exploits that the polynomial representation of the function $U$ is unique, and hence can be written in terms of other basis polynomial functions. Basically, the filtering procedure is:
\begin{enumerate}
\item Transform the coefficients of the nodal approximation into a modal set of basis functions, e.g., the orthogonal Legendre polynomials $\{L_j(\xi)\}_{j=0}^N$
\begin{equation}
U(\xi,t) = \sum_{j=0}^N U_j(t)\ell_j(\xi) = \sum_{j=0}^N\widetilde{U}_j(t)L_j(\xi).
\end{equation}
\item Because the modal basis is hierarchical, it is straightforward for one to perform a cutoff in modal space to filter higher order modes.
\item Transform the filtered modal solution coefficients back to the nodal Lagrange basis.
\end{enumerate}

At present, we select the modal (normalized) Legendre basis polynomials $\{L_j(\xi)\}_{j=0}^N$ to construct the filtering. It is straightforward to compute the Vandermonde matrix $\Vmat$ associated with the LGL nodal interpolation nodes $\{\xi_i\}_{i=0}^N$
\begin{equation}\label{eq:Vandermonde}
\Vmat_{ij} = L_j(\xi_i),\quad i,j = 0,\ldots,N,
\end{equation}
which allows us to transform the nodal degrees of freedom, $\{U_i\}_{i=0}^N$, to modal degrees of freedom, $\{\widetilde{U}_j\}_{j=0}^N$ and vice versa
\begin{equation}
\Nvec{U} = \Vmat\Nvec{\widetilde{U}}\quad\text{and}\quad\Nvec{\widetilde{U}} = \Vinvmat\Nvec{U}.
\end{equation}
The filtering matrix is then constructed as \cite{vandeven1991}
\begin{equation}
\Cmat_{ij} = \delta_{ij}\sigma_i,\quad i,j=0,\ldots,N,
\end{equation}
where $\Cmat$ is a diagonal modal cutoff matrix. The conditions the filter function $\sigma(\eta)$, as defined by Vandeven \cite{vandeven1991}, are 
\begin{itemize}
\renewcommand\labelitemi{\textbullet}
\item $\sigma: \mathbb{R}^+ \mapsto[0,1]$
\item $\sigma(\eta)$ must have $s$ continuous derivatives where
\begin{equation}\label{eq:Vandevenfilter}
\begin{cases}
\sigma(0) = 1&\\[0.05cm]
\sigma^{(k)}(0) = 0,& k = 1,\ldots,s-1\\[0.05cm]
\sigma(\eta) = 0,& \eta\geq 1\\[0.05cm]
\sigma^{(k)}(1) = 0, & k = 1,\ldots,s-1.
\end{cases}
\end{equation}
\end{itemize}

In the nodal DG community a typical choice is an exponential filter function, e.g. \cite{chaudhur2017,Gassner:2013qf,Hestahven:1008th}, to define the coefficients
\begin{equation}\label{eq:filterSigmas}
\sigma_i = \begin{cases}
1 & \textrm{if }0 \leq i \leq N_c -1\\[0.1cm]
\exp\left(-\alpha\left(\frac{i+1-N_c}{N+1-N_c}\right)^s\right) & \textrm{if }N_c \leq i \leq N
\end{cases}
\end{equation}
where $\alpha$, $s$ and $N_c$ are the filter parameters. The value $N_c$ indicates the number of the unaffected modes, $\alpha$ is chosen such that $\exp(-\alpha)$ is machine epsilon and $s$ is an even number determining the order (sometimes referred to as the strength) of the filter. Two common choices for the filtering parameters are to take $\alpha = 36$, $N_c = 4$ and the filter strength is either ``strong'' with $s = 16$ or ``weak'' with $s = 32$ \cite{chaudhur2017,Gassner:2013qf,Hestahven:1008th}. 

For any choice of the filter parameters the filter coefficients are constructed such that $0\leq\sigma_i\leq 1$. Note, this exponential filter does not strictly adhere to Vandeven's definition of the filter function \eqref{eq:filterSigmas}, but it does so in practice by choosing $\alpha$ such that $\sigma(1)$ is below machine accuracy \cite{canuto2006}.

In summary, the filter matrix for the nodal DG approximation is given as
\begin{equation}\label{HWFilter}
\Filtermat = \Vmat\,\Dmat\,\Vinvmat.
\end{equation}
The filter of the form \eqref{HWFilter} retains the high-order accuracy of the nodal DG approximation for smooth functions \cite{hesthaven2008filtering,vandeven1991} as shown numerically in Section~\ref{sec:conv}.

\begin{rem}
Filtering have often been used as a stabilization technique for numerical methods such as in finite difference  \cite{kennedy1997,yee1999} as well as discontinuous Galerkin \cite{chaudhur2017,Gassner:2013qf,Hestahven:1008th} methods. We strongly advise against such use. Instead, one should first construct construct a (provably) stable numerical scheme. After this, the solution quality can be addressed and cleaned-up, possibly using filtering.
\end{rem}

\section{Stability}\label{sec:stableFilterProof}

As previously mentioned, it is possible to develop semi-discrete stability estimates for the nodal DG approximation via the SBP property \eqref{eq:SBP}. 
The filtering is a separate procedure which changes the approximate solution during the time integration procedure, for example after each explicit time step or even after each stage of a Runge-Kutta method. Here, we explore what influence this filtering step has on the stability estimate for the nodal DG approximation.

To discuss the filtering procedure and its affect on stability we re-interpret the work on provably stable filtering from \cite{lundquist2020stable} into the nodal DG context. 
In a broad sense, with homogeneous boundary conditions, semi-discrete stability ensures that the discrete norm of the approximate solution is bounded by the discrete norm of the initial conditions, see \cite{Nordstrom:2016jk} for complete details. For the nodal DG approximation such a stability statement takes the form
\begin{equation}
\inormN{U(t)} \leq \inormN{U_{\text{ini}}},
\end{equation}
where $U_{\text{ini}}$ is the initial condition evaluated at the LGL nodes.

Pursuant to the work \cite{lundquist2020stable}, we view the application of a filter matrix to a discrete solution at some intermediate time $t_1$ as a transmission problem \cite{nordstrom2018well}:
\begin{equation}\label{eq:transmissionProblem}
\begin{aligned}
\Nvec{U}_t + \mathbb{D}(\Nvec{U}) &= 0,\qquad 0 \leq t \leq t_1\\[0.1cm]
\Nvec{V}_t + \mathbb{D}(\Nvec{V}) &= 0,\qquad t \geq t_1\\[0.1cm]
\Nvec{U}(0) &= \Nvec{U}_{\text{ini}}\\[0.1cm]
\Nvec{V}(t_1) &= \Filtermat\Nvec{U}(t_1)\\[0.1cm]
\end{aligned}
\end{equation}
where the operator $\mathbb{D}$ contains the derivative matrix $\Dmat$ as well as the boundary conditions. For the present discussion, the filtering stated in the final line of \eqref{eq:transmissionProblem} is performed in an explicit fashion.

For stability it must hold that the filter is \textit{contractive}, i.e.
\begin{equation}\label{eq:discNormContract}
\inormN{V(t_1)} \leq \inormN{U(t_1)}.
\end{equation}
In turn, this contractive property guarantees that the filter procedure is \textit{stable} because
\begin{equation}
\inormN{V(t_1)} \leq \inormN{U(t_1)} \leq \inormN{U_{\text{ini}}}.
\end{equation}

The contractivity property in the discrete norm \eqref{eq:discNormContract} then implies that the following contractivity condition on the explicit filter matrix $\Filtermat$ must hold
\begin{equation}\label{eq:contract}
\Filtermat^T\,\Mmat\,\Filtermat - \Mmat \leq 0.
\end{equation}
This contractive matrix property was first identified in \cite{nordstrom2018well} for stable transmission problems. The contractivity condition \eqref{eq:contract} expresses a precise interplay between the filter matrix and the mass matrix. As demonstrated in \cite{lundquist2020stable}, a necessary condition for the explicit filter matrix to satisfy \eqref{eq:contract} is that an auxiliary filter matrix
\begin{equation}\label{eq:auxFilter}
\explicitFiltermat = \Mmat^{-1}\Filtermat^T\Mmat,
\end{equation}
exists and possesses the same accuracy as $\Filtermat$. The accuracy requirement on $\explicitFiltermat$ is necessary since otherwise \eqref{eq:contract} is provably indefinite \cite{lundquist2020stable}. 


We will show that the auxiliary filter matrix \eqref{eq:auxFilter} is identical to the original filter matrix \eqref{HWFilter} for the LGL nodal DG approximation. Furthermore, the filter matrix $\Filtermat$ indeed satisfies the contractivity condition \eqref{eq:contract}. Both results require the following Lemma.
\begin{lemma}\label{lem:quad}
The matrix product $\Vmat^T\Mmat\,\Vmat$ is the LGL quadrature rule applied to the (normalized) Legendre polynomial functions $\{L_j(\xi)\}_{j=0}^N$ and results in a diagonal matrix
\begin{equation}\label{eq:quadMatDef}
\Vmat^T\Mmat\,\Vmat = \textnormal{diag}\left(1,1,\ldots,1,2 + \frac{1}{N}\right) \coloneqq \quadmat.
\end{equation}
\end{lemma}
\begin{proof}
The entries of this matrix product in terms of discrete inner products is
\begin{equation}\label{eq:quadVandermonde}
\Vmat^T\Mmat\,\Vmat = \begin{bmatrix}
\inormN{L_0}^2 & \iprodN{L_0,L_1} & \cdots & \iprodN{L_0,L_N} \\[0.05cm]
\iprodN{L_1,L_0} & \inormN{L_1}^2 & \cdots & \iprodN{L_1,L_N} \\[0.05cm]
\vdots & \vdots & \ddots & \vdots\\[0.05cm]
\iprodN{L_N,L_0} & \iprodN{L_N,L_1} & \cdots & \inormN{L_N}^2 \\[0.05cm]
\end{bmatrix}.
\end{equation}
From the accuracy of the LGL quadrature and the fact that $\{L_j(\xi)\}_{j=0}^N$ are polynomials, we have equality between the discrete and continuous inner products
\begin{equation}\label{eq:LegendreInner}
\iprodN{L_j,L_k} = \iprod{L_j,L_k} = \inorm{L_j}^2\delta_{jk} = \delta_{jk},
\end{equation}
provided $j+k \leq 2N-1$. The result above utilizes that the Legendre basis is orthonormal. The quadratures in \eqref{eq:quadVandermonde} are therefore exact for all inner products except the one related to $L_N$ with itself since $2N > 2N-1$. Thus,
\begin{equation}\label{eq:preK}
\Vmat^T\Mmat\,\Vmat = \text{diag}(1,1,\ldots,1,\inormN{L_N}^2).
\end{equation}
The discrete and continuous norms are equivalent \cite{canuto2006}. Provided that $\phi$ is a polynomial of degree $N$, the discrete and continuous $L^2$ norms are related by $\inormN{\phi} = \sqrt{2 + 1/N}\inorm{\phi}$ for the LGL quadrature  \cite{ISI:A1982NE30900005} .
Using this fact, \eqref{eq:preK} becomes the diagonal matrix $\quadmat$.
\end{proof}

We can now prove
\begin{proposition}\label{prop:auxFilter}
The auxiliary filter is identical to the DG filter matrix, i.e. $\explicitFiltermat = \Filtermat$.
\end{proposition}
\begin{proof}
We examine the difference between the two filter matrices
\begin{equation}
\explicitFiltermat - \Filtermat = \Mmat^{-1}\Filtermat^T\Mmat - \Filtermat = \Mmat^{-1}\mathcal{V}^{-T}\Cmat\,\Vmat^T\Mmat - \Vmat\,\Cmat\,\Vinvmat.
\end{equation}
Next, we factor out the matrix $\Vmat$ on the left and $\Vinvmat$ on the right to have
\begin{equation}
\begin{aligned}
\explicitFiltermat - \Filtermat &=\Vmat\left[\left(\Vinvmat\Mmat^{-1}\mathcal{V}^{-T}\right)\Cmat\left(\Vmat^T\Mmat\Vmat\right) - \Cmat\right]\Vinvmat\\[0.15cm]
&=\Vmat\left[\left(\Vmat^T\Mmat\Vmat\right)^{-1}\Cmat\left(\Vmat^T\Mmat\Vmat\right) - \Cmat\right]\Vinvmat.
\end{aligned}
\end{equation}
Applying the result from Lemma~\ref{lem:quad} gives
\begin{equation}
\explicitFiltermat - \Filtermat =\Vmat\left[\quadmat^{-1}\Cmat\quadmat - \Cmat\right]\Vinvmat =\Vmat\left[\quadmat^{-1}\quadmat\Cmat - \Cmat\right]\Vinvmat = 0,
\end{equation}
where we use that the matrices $\Cmat$ and $\quadmat$ are diagonal to obtain the desired result.
\end{proof} 
\begin{rem}
The accuracy of the filter $\Filtermat$ lies entirely in the filter function $\sigma$ used to create the diagonal entries in the matrix $\Cmat$, as shown by Vandeven \cite{vandeven1991}. 
\end{rem}

The following result is then self-evident.

\begin{corollary}
The auxiliary filter $\explicitFiltermat$ exists and is as accurate as $\Filtermat$.
\end{corollary}

Further, the result from Lemma~\ref{lem:quad} allows us to prove
\begin{proposition}\label{prop:contract}
The nodal DG filter matrix $\Filtermat$ is contractive in the sense of \eqref{eq:contract}. 
\end{proposition}
\begin{proof}
We substitute the form of the filter matrix \eqref{HWFilter} into the contractivity condition \eqref{eq:contract} to obtain
\begin{equation}\label{eq:coerce}
\Filtermat^T\,\Mmat\,\Filtermat - \Mmat = (\Vmat\,\Cmat\,\Vinvmat)^T\,\Mmat\,(\Vmat\,\Cmat\,\Vinvmat) - \Mmat
 = (\Vinvmat)^T\Cmat\,(\Vmat^T\Mmat\,\Vmat)\,\Cmat\,\Vinvmat - \Mmat.
\end{equation}
The middle term, $\Vmat^T\Mmat\,\Vmat$, grouped above is precisely that from Lemma~\ref{lem:quad}, which gives
\begin{equation}\label{eq:coerce2}
\Filtermat^T\,\Mmat\,\Filtermat - \Mmat = (\Vinvmat)^T\Cmat\,\quadmat\,\Cmat\,\Vinvmat - \Mmat.
\end{equation}
Next, we recall that the modal cutoff matrix is diagonal with entries $\{\sigma_i\}_{i=0}^N$ and, by construction, the term $\sigma_N = 0$, see \eqref{eq:Vandevenfilter}. Thus, 
\begin{equation}\label{eq:diagCmat}
\Cmat = \text{diag}(\sigma_0,\sigma_1,\ldots,\sigma_{N-1},0).
\end{equation}
We combine this fact with the diagonal quadrature matrix result \eqref{eq:quadMatDef} to simplify the middle term of \eqref{eq:coerce2} to be $\Cmat\,\quadmat\,\Cmat = \Cmat^2$. The contractivity condition then becomes
\begin{equation}
\Filtermat^T\,\Mmat\,\Filtermat - \Mmat = (\Vinvmat)^T\Cmat^2\Vinvmat - \Mmat = (\Vinvmat)^T\left(\Cmat^2 - \quadmat\right)\Vinvmat,
\end{equation}
where we, again, apply the result from Lemma~\ref{lem:quad}. From \eqref{eq:quadMatDef} and \eqref{eq:diagCmat} we see that
\begin{equation}
\Cmat^2 - \quadmat = \text{diag}\left(\sigma^2_0-1,\sigma^2_1-1,\ldots,\sigma^2_{N-1}-1,-2-\frac{1}{N}\right) \leq 0,
\end{equation}
because $0\leq\sigma_i\leq 1$ for $i=0,\ldots,N$. 
Thus, \eqref{eq:contract} holds.
\end{proof}


\begin{rem}
The proof above holds provided the filter function is chosen such that $\sigma(\eta)\in[0,1]$ and that it ``clips'' the highest mode, i.e. $\sigma(N) = 0$, then the nodal DG filter matrix $\Filtermat$ satisfies the contractivity condition \eqref{eq:contract}. Thus, other proposed filter functions like those found in Hussaini et al. \cite{hussaini1985spectral} also produce a provably stable filter matrix.
\end{rem}

\section{Numerical results}\label{sec:numResults}

Here we apply the nodal DG filter matrix in $\Filtermat$ \eqref{HWFilter} to several test problems. For these tests we select the parameters for a ``strong'' version of the nodal DG filter $\Filtermat$ from Section~\ref{sec:standardDG}. We do so to demonstrate, in practice, the high-order accuracy of the filtered DG approximation for a smooth solution and how it performs for test cases that develops non-smooth solutions over time. To integrate the semi-discrete DG approximation \eqref{eq:semiDG} in time, we use a third-order Runge-Kutta from Williamson \cite{williamson1980}.

\subsection{High-order convergence for linear advection}\label{sec:conv}

For the convergence test we consider the linear advection equation with flux function $f(u)=a u$ and take the wave speed to be the constant $a=1$ on the domain $\Omega = [0,1]$. 

We use a smooth Gaussian pulse to set the initial and boundary conditions as well as compute errors
\begin{equation}
u(x,t) = \exp\left(-\zeta\left(x - 0.25 - t\right)^2\right),\quad\text{with}\quad \zeta = \frac{\ln(2)}{0.2^2}.
\end{equation}
We vary the polynomial degree $N$ at values in the interval $[7,64]$ and integrate up $T=0.5$ as the final time. In Fig.~\ref{fig:spec_conv} we present a semilog plot of the $L_{\infty}$ error versus the polynomial order. We see that the error decays exponentially until the errors in the approximation are dominated by the time integration. Thereafter, we halve the time step size and see that the stagnation point of the error drops by a factor of eight, as expected for a third-order time integration technique.
\begin{figure}[htbp]
\begin{center}
{
{\includegraphics[scale=0.49, trim=8 11 10 10, clip]{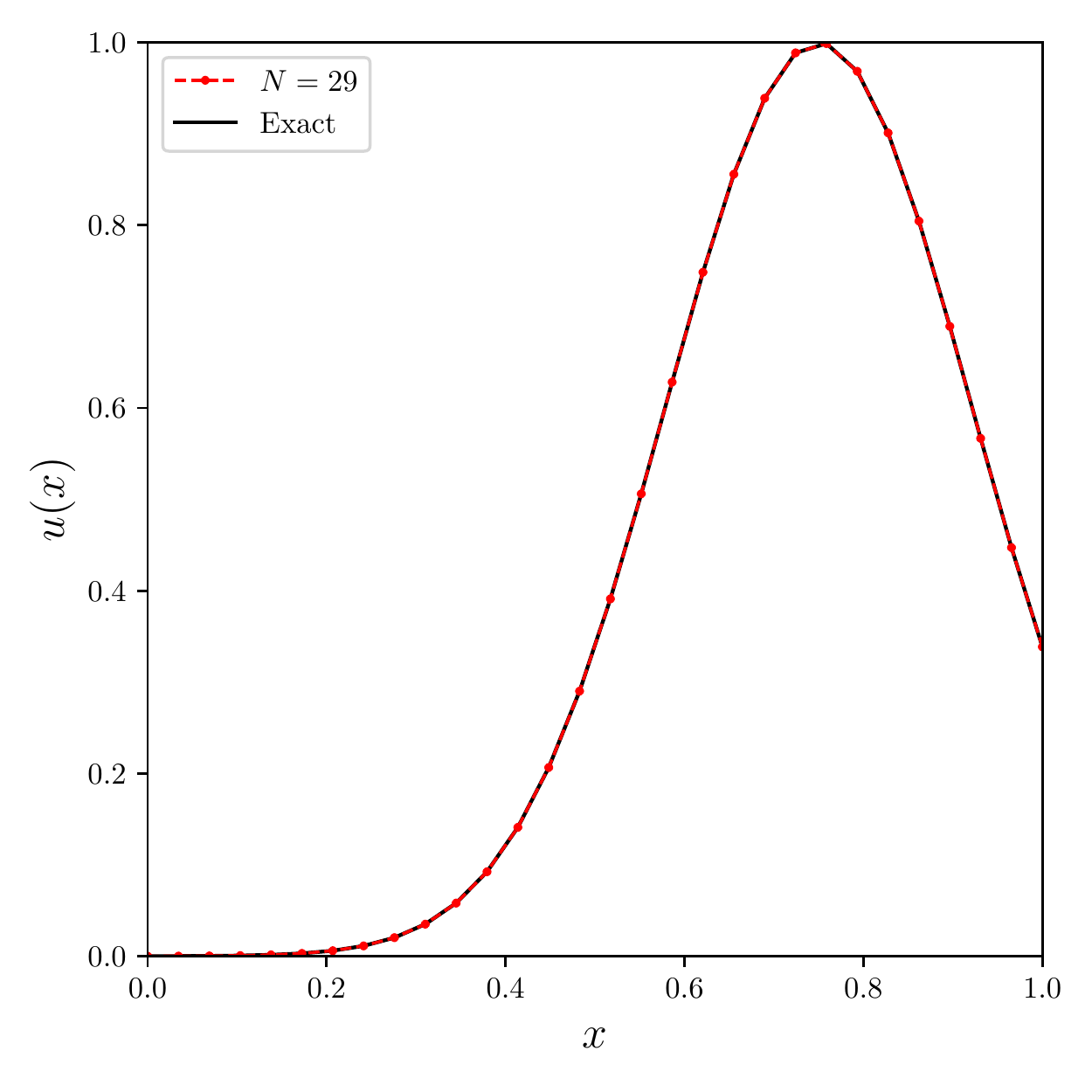}}
}
{
{\includegraphics[scale=0.34, trim=30 170 65 180, clip]{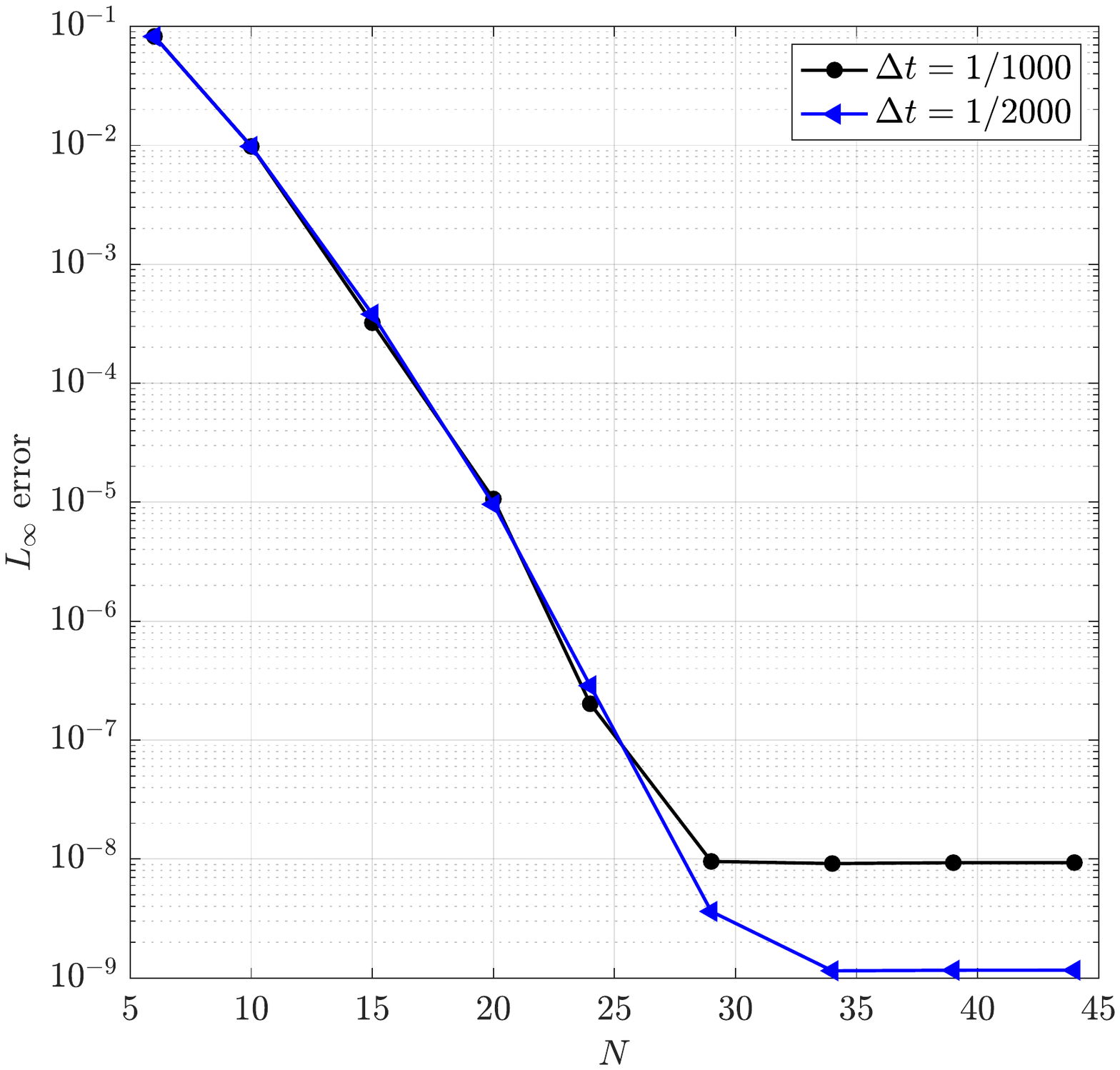}}
}
\caption{(\textit{left}) Exact and approximate solution with $N=29$ for the linear advection with smooth Gauss pulse at $T=0.5$. (\textit{right}) Space and time convergence.}
\label{fig:spec_conv}
\end{center}
\end{figure}

This experimentally demonstrates that the nodal DG approximation filtered in every time step remains high-order accurate in space and time for a smooth solution.

\subsection{Variable wave speed for linear advection}

Next, we consider a more complicated test proposed in \cite{hesthaven2008filtering}. For this case, the solution remains bounded, but develops steep gradients. Due to the high-degree polynomial approximation of the DG method, spurious oscillations can develop near these steep gradients and propagate throughout the domain, polluting the solution quality. 

We consider the linear advection problem with a variable wave speed on the domain $\Omega=[-1,1]$ written in the form
\begin{equation}\label{eq:varCoeff}
\pderivative{u}{t} + a(x)\pderivative{u}{x} = 0,\quad\text{where}\quad a(x) = \frac{1}{\pi}\sin(\pi x - 1).
\end{equation}
This wave speed remains positive at the boundaries of the domain, but it can change sign within the domain. We take the initial condition to be 
$u_{\text{ini}}(x) = \sin(\pi x)$
which gives the corresponding analytical solution \cite{gottlieb1981stability}
\begin{equation}\label{eq:varCoeffAnalytical}
u(x,t) = \sin\left(2\tan^{-1}\left(e^{-t}\tan\left(\frac{\pi x -1}{2}\right)\right) + 1\right).
\end{equation}
The solution in \eqref{eq:varCoeffAnalytical} develops a steep gradient around the point $x = (1-\pi)/\pi\approx -0.68$ before it finally decays to a constant value 
$\lim_{t\to\infty}u(x,t) = \sin(1)$.

We compare the analytical and approximate solutions at a final time $T=4$ on a single spectral element using polynomial order $N=256$ and $\Delta t = 1/2000$ as the explicit time step. The polynomial order and general setup are chosen such that a comparison can be made to the results in \cite{hesthaven2008filtering}. In Figure \ref{fig:varCoeff} we present the unfiltered approximation on the left and the filtered approximation on the right. Clearly, the unfiltered DG scheme contains spurious oscillations whereas the filtered solution suppresses such behaviour. Also, qualitatively, the results of the new DG filter scheme are very similar to those from \cite{hesthaven2008filtering}.
\begin{figure}[htbp]
\begin{center}
{
{\includegraphics[scale=0.49, trim=10 12 10 10, clip]{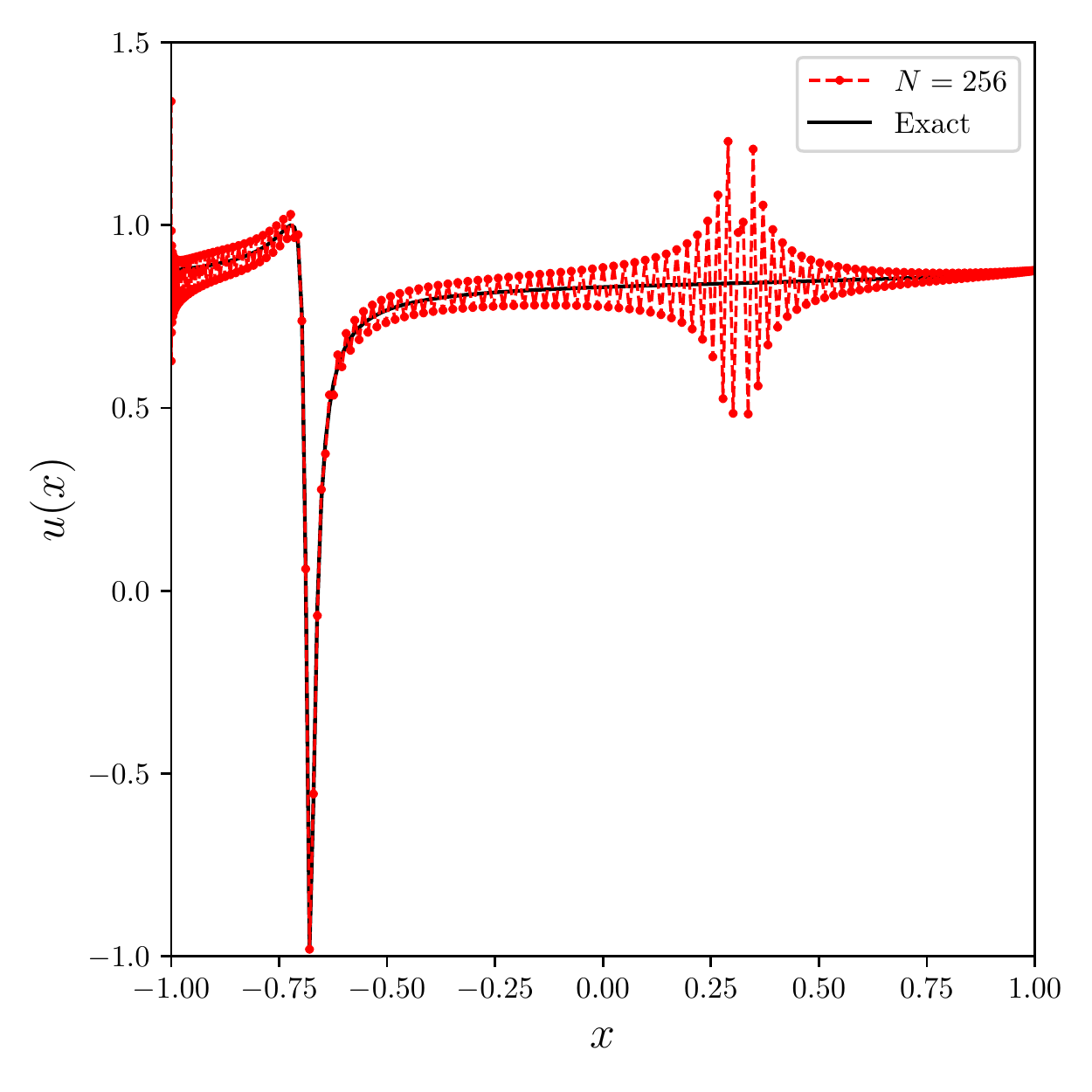}}
}
{
\includegraphics[scale=0.49, trim=35 12 10 10, clip]{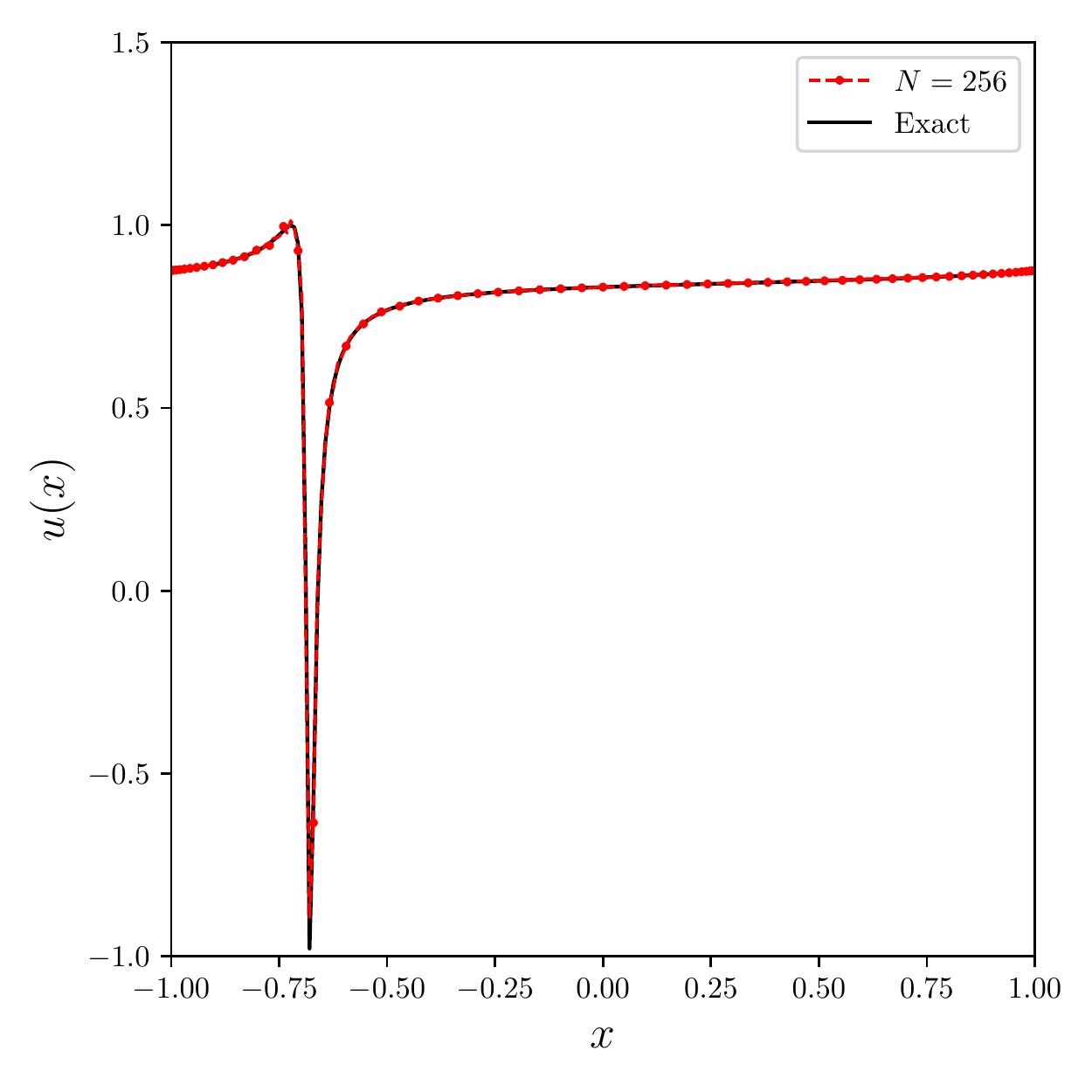}
\vspace{-0.3cm}
}
\caption{Unfiltered (\textit{left}) and filtered (\textit{right}) nodal DG solution for the variable wave speed problem \eqref{eq:varCoeff} at $T=4$ with polynomial order $N=256$ on a single spectral element.}
\label{fig:varCoeff}
\end{center}
\end{figure}

\subsection{Stability demonstration for Burgers'}

This example is designed to illustrate the importance of the stability of the underlying spatial discretization and how the filtering can influence the behavior of the solution in time. For this we consider Burgers' equation and two forms of the nodal DG spatial discretization. One discretizes the conservative form of the PDE where the nonlinear Burgers' flux is simply $f(u) = u^2/2$ while the other skew-symmetric discretization writes the spatial derivative of the flux in a split formulation
\begin{equation}
f^{\text{skew}}_x(u) = \frac{2}{3}\left(\frac{u^2}{2}\right)_{\!x} + \frac{1}{3}\left(uu_x\right).
\end{equation}
On the continuous level these two forms of Burgers' equation are equivalent; however, on the discrete level they exhibit different behavior. Most notably, the solution energy $u^2/2$ is bounded for the nodal DG discretization constructed from the skew-symmetric formulation whereas no such bound exists for the discretization of the conservative form, see \cite{gassner_skew_burgers} for complete details. 

Therefore, the discretization of the skew-symmetric form is provably stable and the conservative form \textit{is not}. As discussed in the previous Sections, the DG filtering is a procedure divorced from the spatial discretization. If the underlying numerical scheme is provably stable, and the filtering is contractive, it will remove energy from the solution in a stable way. However, if the underlying scheme is unstable the filtering still removes energy and the approximation \textit{might} be stable, but no further conclusions regarding the solution energy can be drawn.

To illustrate this we consider the domain $\Omega=[0,2]$ with periodic boundary conditions and the initial condition
\begin{equation}\label{eq:burgersinitial}
u_{\text{ini}}(x) = \frac{1}{5}[1 + \cos(\pi x)].
\end{equation}
We run the simulation with polynomial order $N=128$ up to $T=2.25$ as a final time. Further, the solution is filtered at 16 equally spaced times during the simulation. Due to the nonlinear nature of Burgers' equation the initial conditions will steepen and eventually a shock will form.

We run four variants of the nodal DG scheme relevant to the present discussion:
\begin{itemize}[label=$\bullet$]
\item Conservative formulation; Unfiltered.
\item Conservative formulation; Filtered.
\item Skew-symmetric formulation; Unfiltered.
\item Skew-symmetric formulation; Filtered.
\end{itemize}
On the left in Figure~\ref{fig:burgers} we present the evolution of the solution energy, normalized with its initial value, over time. On the right in Figure~\ref{fig:burgers} we give the approximate solution at the final time produced by the filtered skew-symmetric DG formulation as well as a reference solution created with a standard finite volume scheme on 10000 grid cells.
\begin{figure}[htbp]
\begin{center}
{
\includegraphics[scale=0.49, trim=5 12 10 12, clip]{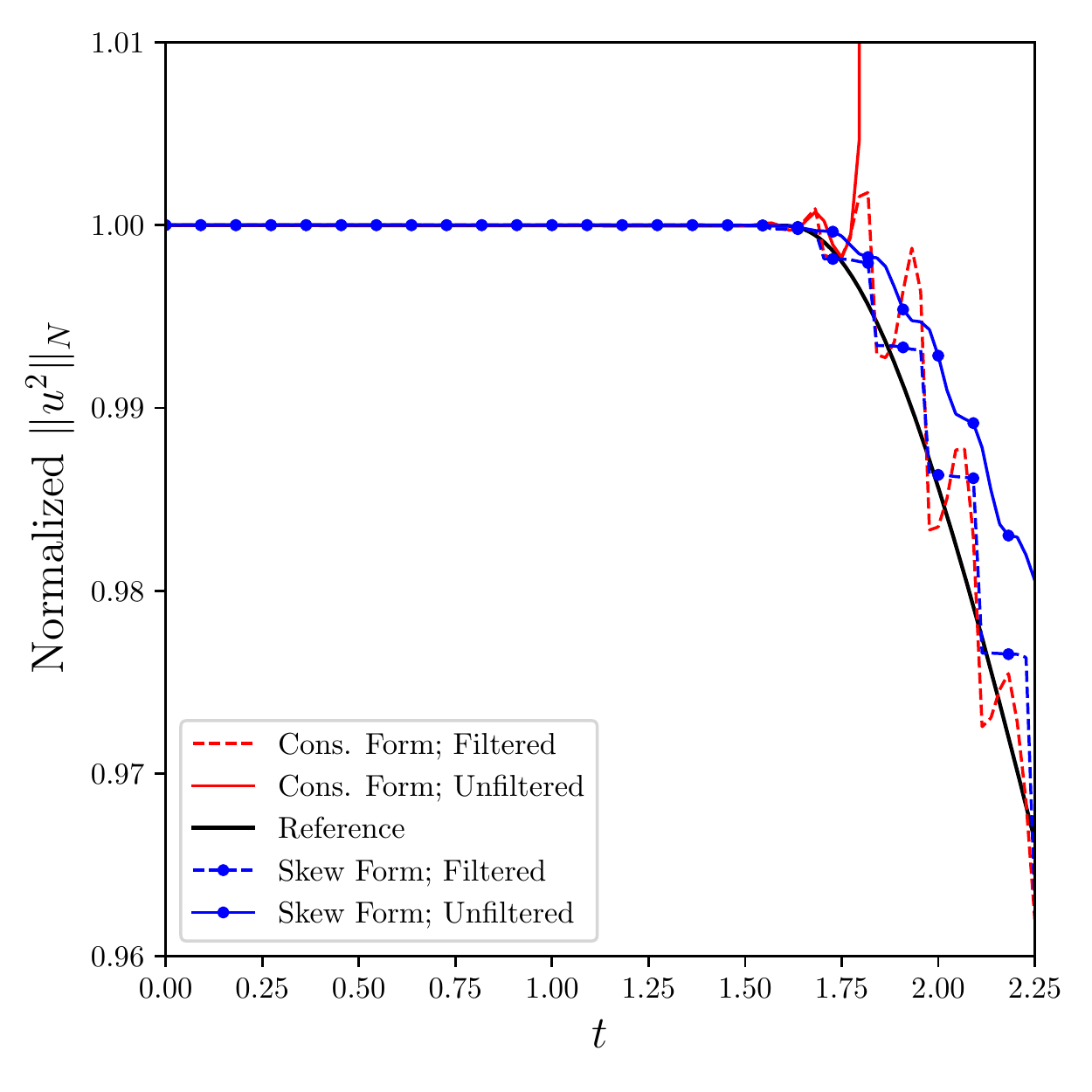}
}
{
\includegraphics[scale=0.49, trim=10 12 10 12, clip]{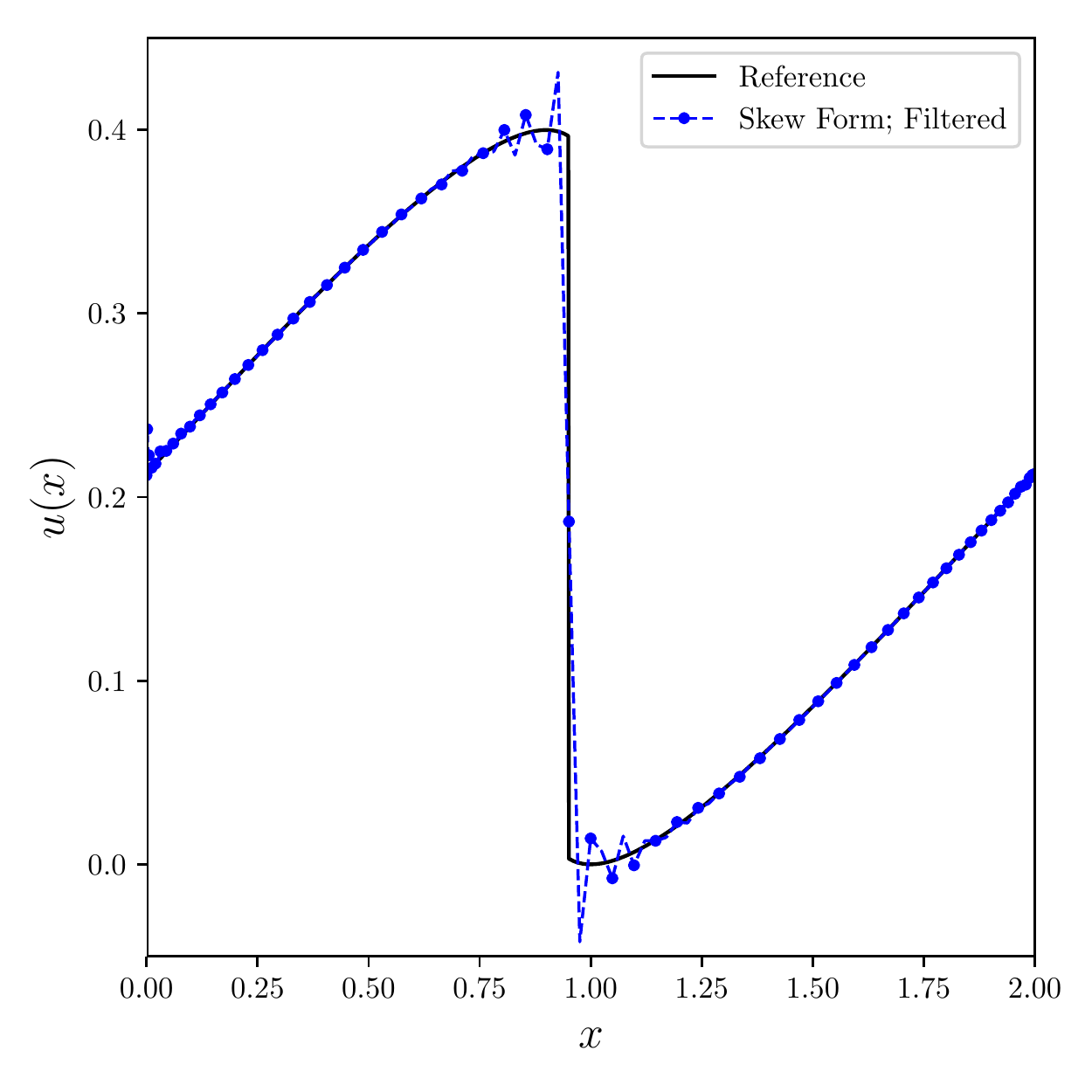}
}
\caption{(\textit{left}) Solution energy evolution of four nodal DG variants for Burgers' equation with the initial condition \eqref{eq:burgersinitial}. (\textit{right}) Plot of the filtered skew-symmetric solution ($N=128$) and a reference finite volume solution at $T=2.25$.}
\label{fig:burgers}
\end{center}
\end{figure}

Due to the high polynomial order, the simulation is well resolved and the four variants are nearly indistinguishable for most of the simulation time. However, as the gradients steepen we see that the conservative formulation, for which no energy stability statement exists, behave erratically. The unfiltered conservative form simulation crashes at $T\approx 1.8$. The filtered conservative form simulation successfully runs as the filtering keeps the solution energy ``under control.'' This is illustrated in Figure~\ref{fig:burgers} where we observe growth in the solution energy between the filter applications because the underlying spatial discretization is unstable. The solution energy of unfiltered and filtered skew-symmetric simulations both remain bounded because the underlying spatial discretization possesses an energy bound. Note, there is a small amount of dissipation in the solution energy for the unfiltered skew-symmetric scheme due to the formation of the shock \cite{jameson2008_energy}. Further, the filtered skew-symmetric simulation is less energetic, as expected, because the act of filtering removes some solution energy.
\vspace{-0.0cm}
\section{Closing remarks}

We proved that the commonly used nodal DG filter matrix $\Filtermat$ satisfied a contractivity condition. Further, we proved that a high-order auxiliary filter matrix $\explicitFiltermat$ exists for the nodal DG approximation which is necessary for the contractivity condition to be satisfied. Together, these results implied that the explicit filtering procedure in the context of nodal DG methods ``removes'' information in a stable way when measured in the norm induced by the Legendre-Gauss-Lobatto quadrature. 

Numerical results were provided to demonstrate and verify that the filtering retained the high-order accuracy of the nodal DG approximation, that it suppressed spurious oscillations near steep gradients and was stable provided the underlying spatial discretization of the method had a semi-discrete bound.

The generalization of the results described in this work to multiple spatial dimensions is straightforward due to the tensor product nature of the nodal DG method. The same is true for problems involving curved physical boundaries, similar to those discussed in \cite{lundquist2020stable}. An interesting open question for future work is the extension of provably stable filtering to problems involving multiple DG elements where the interface coupling will play a key role.

\section{Declarations}

\subsection{Funding}
This work was supported by Vetenskapsr{\aa}det, Sweden (award number 2018-05084 VR)

\subsection{Conflicts of interest}
The authors declare that they have no conflicts of interest in the present work.

\subsection{Availability of data and material}
Not applicable.

\subsection{Code availability}
The code used to generate the results in this work is available upon request with Andrew Winters (andrew.ross.winters@liu.se).

\bibliographystyle{spmpsci}
\bibliography{final_build}

\end{document}